\documentclass{amsart}

\usepackage{pseudocode, amssymb, graphicx}

\newtheorem{theorem}{Theorem}[section]
\newtheorem{lemma}[theorem]{Lemma}

\theoremstyle{definition}

\newtheorem{Prop}[theorem]{Proposition}
\newtheorem{example}[theorem]{Example}

\newtheorem{conjecture}[theorem]{Conjecture}

\theoremstyle{remark}

\numberwithin{equation}{section}

\DeclareMathOperator{\lcm}{lcm}

\begin{document}

\title{Strong Pseudoprimes to Twelve Prime Bases}


\author{Jonathan Sorenson}
\address{Department of Computer Science and Software Engineering, Butler University}
\curraddr{}
\email{sorenson@butler.edu}
\thanks{}

\author{Jonathan Webster}
\address{Department of Mathematics and Actuarial Science, Butler University}
\curraddr{}
\email{jewebste@butler.edu}
\thanks{This work was supported in part by a grant from the Holcomb Awards Committee.
   We thank Frank Levinson for his generous support of Butler University's
    cluster supercomputer, \textit{Big Dawg}.
We would also like to thank Andrew Shallue for his feedback on a draft of this paper,
  and Paul Pollack for his help with some details regarding an earlier proof
  approach to the running time bound for $\lambda$-sieving.}

\date{\today}

\dedicatory{}

\begin{abstract} 
Let $\psi_m$ be the smallest strong pseudoprime to the first $m$ prime bases. 
This value is known for $1 \leq  m \leq 11$. 
We extend this by finding $\psi_{12}$ and $\psi_{13}$. 
We also present an algorithm to find all integers $n\le B$
  that are strong pseudoprimes to the first $m$ prime bases;
  with a reasonable heuristic assumption we can show
  that it takes at most $B^{2/3+o(1)}$ time.
\end{abstract}

\maketitle

\section{Introduction}

Fermat's Little Theorem states that if $n$ is prime and $\gcd (a,n) = 1$,
   then \[ a^{n-1} \equiv 1 \pmod{n}.\]
A composite number for which this congruence holds is called   
  a \textit{pseudoprime} to the base $a$.  
We can restate Fermat's Little Theorem by algebraically factoring   
  (repeatedly) the difference of squares that arises in $a^{n-1} - 1$. 
In which case, if $n$ is an odd prime, writing $n - 1 = 2^sd$ with $d$ odd,  
  and $\gcd (a,n) = 1$, then either
\[ a^d \equiv 1 \pmod{n}\]
or
\[ a^{2^kd} \equiv -1 \pmod{n} \]
for some   integer $k$ with $0 \leq k < s$. If a composite $n$ satisfies the above,   we call $n$ a \textit{strong pseudoprime} to the base $a$.  Unlike Carmichael numbers, which are pseudoprimes to all bases,   there do not exist composite numbers that are strong pseudoprimes to all bases. This fact forms the basis of the Miller-Rabin probabilistic   primality test \cite{rabin}. 
Define $\psi_m$ to be the smallest integer that is a strong pseudoprime to the first $m$ prime bases.
(See A014233 at \texttt{oeis.org}.)

The problem of finding strong pseudoprimes has a long history.  Pomerance, Selfridge, and Wagstaff \cite{pom_self} computed   $\psi_m$ for $m = 2, 3,$ and $4$. Jaeschke \cite{jaeschke} computed   $\psi_m$ for $m = 5, 6, 7,$ and $8$. By looking at a narrow class of numbers,  Zhang \cite{zhang} gave upper bounds on $\psi_m$   for $9 \leq m \leq 19$.  Recently, Jian and Deng \cite{8bases} verified some of Zhang's   conjectures and computed $\psi_m$ for $m = 9, 10,$ and $11$.  We continue in this effort with the following.

\begin{theorem}
\begin{align*}
\psi_{12} &=  3186\ 65857\ 83403\ 11511\ 67461.\\
\psi_{13} &=  33170\ 44064\ 67988\ 73859\ 61981.
\end{align*}
\end{theorem}

We also verified Jian and Deng's work in finding   $\psi_{9}= \psi_{10}= \psi_{11}$.
Note that the ERH implies $\psi_m\ge \exp(\sqrt{m/2})$, if $\psi_m$ exists
  \cite{Bach1990}.

The proof of our theorem relies primarily on running an 
  improved algorithm for finding strong pseudoprimes:

\begin{theorem} \label{thm:runtime}
Given a bound $B>0$ and an integer $m>0$ that grows with $B$,
  there is an algorithm to find all integers $\le B$ that  
  are products of exactly two primes and
  are strong pseudoprimes to the first $m$ prime bases 
  using at most $B^{2/3+o(1)}$ arithmetic operations. 
\end{theorem}

We have a heuristic argument extending this $B^{2/3+o(1)}$
  running time bound to integers with an arbitrary number of prime divisors.
This result is an improvement over a heuristic $O(B^{9/11})$ from \cite{bleich}.  
Our algorithm uses the work of Jian and Deng as a starting point,  and our specific improvements are these:
\begin{itemize}
  \item We applied the results in Bleichenbacher's thesis \cite{bleich}    to speed the computation of GCDs.  See \S\ref{sec:gcd}.
  \item We applied the space-saving sieve \cite{sorensonagain, sorensonanother}  which enabled us to sieve using full prime signature information.   See \S\ref{sec:sigsieve}.
  \item We used a simple table of linked-lists indexed by a hash function based on prime signatures.   We used this to store small primes by signature for very fast lookup.    See \S\ref{sec:datastructure}.
  \item When it made sense to do so, we parallelized our code and made use of \textit{Big Dawg},  Butler University's cluster supercomputer.
\end{itemize}
Our changes imply that adding more prime bases  for a pseudoprime search   (that is, making $m$ larger) speeds up the algorithm.

The rest of this paper is organized as follows.  Section 2 contains a description of the algorithm and   its theoretical underpinnings.   Section 3 contains the proof of the running time.  Section 4 contains implementation details.

\section{Algorithmic Theory and Overview}

To find all $n\le B$ that are strong pseudoprimes to the first $m$  prime bases, we use some well-known conditions to limit how many  such integers $n$ we need to check. This is outlined in the first subsection below.

We construct all possibilities by generating $n$ in factored form, $n=p_1p_2\ldots p_{t-1} p_t$,  where $t$ is the number of prime divisors of $n$, and $p_i<p_{i+1}$. It turns out that $t$ cannot be too big; we only had to test up to $t=6$. So we wrote separate programs for each possible $t$ and the case $t=2$ dominates the running time.

Let $k=p_1p_2\cdots p_{t-1}$. Our algorithms generate all candidate $k$ values,   and then for each $k$ search for possible primes $p_t$ to match with $k$,   to form a strong pseudoprime candidate $n=kp_t$.

For the bound $B$, we choose a cutoff value $X$. For $k\le X$, we use a GCD computation  to search for possible primes $p_t$, if they exist. For $k>X$, we use sieving to find possible primes $p_t$. Again, we use separate programs for GCDs and for sieving.

Once each $n$ is formed, we perform a strong pseudoprime test for  the first $m$ prime bases to see if we have, in fact, found  what we are looking for.

\subsection{ Conditions on (Strong) Pseudoprimes}

For the purposes of this paper, it suffices to check square-free odd integers. A computation by Dorais and Klyve \cite{FandD} shows that a   pseudoprime to the bases 2 and 3 must be square-free   if it is less than $4.4\cdot 10^{31}$.  Since the values Zhang conjectured as candidates for $\psi_{12}$   and $\psi_{13}$ are less than $4.4 \cdot 10^{31}$, we restrict our search to square-free numbers.  

In general, to rule out integers divisible by squares, it suffices to  find all Wieferich primes $\le B^{1/2}$, and for each such prime,  perform appropriate tests. This can easily be done in well under $O(B^{2/3})$ time. See \S3.6.1 in \cite{bleich}.

To each prime $p$ we associate a vector called its \textit{signature}. Let $\nu = (a_1, a_2, \ldots, a_m)$ for $a_i$ distinct positive integers.  Then the signature of $p$ is 
\[ \sigma_p^{\nu} = \left( v(\mbox{ord}_p(a_1)), v(\mbox{ord}_p(a_2)), \ldots, v( \mbox{ord}_p(a_m)  )  \right),\]
where $v$ denotes valuation with respect to $2$ and $\mbox{ord}_p(a)$ is the multiplicative order of $a$ modulo $p$. 

\begin{example} 
  Let $p$ be the prime $151121$.  Then   $\sigma_p^\nu = (3,4,0,4,2,1,2,4)$  if $\nu=(2,3,5,7,11,13,17,19)$.
\end{example} 

\begin{theorem}[Proposition 1 in \cite{jaeschke}]
Let $n = p_1 \cdots p_t$ with $t$ distinct primes, $\nu = (a_1, a_2, \ldots, a_m)$ with different integers such that $p_i \nmid a_j$ for all $1 \leq i \leq t$ and $1 \leq j \leq m$.  Then $n$ is a strong pseudoprime to each base in $\nu$ if and only if n is a pseudoprime to each base in $\nu$ and $\sigma_{p_1}^{\nu} =  \cdots = \sigma_{p_t}^{\nu}$.
\end{theorem}

If $t > 2$, then this theorem greatly limits the number of  candidate $k$ values we need to check. In the case $t=2$ the initial sub-product $k$ is just a single prime, so this theorem does not, at first, appear to help. However, it does play a role in sieving for $p_t$. For the rest of this paper,   we use $\nu$ as the vector containing the first $m=11$ or $12$ prime numbers.  

\subsection{ Greatest Common Divisors\label{sec:gcd}}

As above, let $k = p_1\cdots p_{t-1}$.  We explain in this subsection how to find all possible choices for $p_t$  to form $n=kp_t$ as a strong pseudoprime to the first $m$ prime bases  using a simple GCD computation.

Let $b$ be one of the $a_i$ from $\nu$. The pseudoprimality condition  $b^{kp_t - 1} \equiv 1 \pmod{ kp_t}$ implies that

\[  b^{kp_t - 1} \equiv b^{k(p_t -1) + k - 1} \equiv b^{k-1} \equiv 1 \pmod{p_t}, \]

i.e., that $p_t$ divides $b^{k-1} - 1$.  Considering multiple bases, we know that $p_t$ must divide 

\[ \gcd \left( a_1^{k-1} - 1,  a_2^{k-1} - 1, \ldots,  a_m^{k-1} - 1 \right) .\]

However, since we are concerned with strong pseudoprimes,   we can consider only the relevant algebraic factor of $b^{k-1} -1$.  Following Bleichenbacher's notation \cite{bleich},   let $k - 1 = u2^d$ for $u$ odd and

\[ h(b, k) = \left\{ \begin{array}{ll}   b^u - 1  & \mbox{ if } v( \mbox{ord}_k(b)) =0 \\     b^{u2^{c-1}} + 1 &  \mbox{ if }v( \mbox{ord}_k(b)) = c > 0   \end{array}   \right. .  \]
             
\begin{theorem}[Theorem 3.23 of \cite{bleich}]
A necessary condition for an integer $n$ of the from $n = kp_t$,   for a given $k$ with $p_t$ prime,   to be a strong pseudprime to the bases $a_1, \ldots, a_m$ is that  $p_t$ divide $\gcd( h(a_1,k), \ldots, h(a_m, k) )$.
\end{theorem} 

This theorem allows us to construct the product of all  possible $p_t$ values for a given $k$ by computing a GCD. To carry this out, we figure out which of the $h(a_i,k)$ values  would be the smallest, compute that, and then compute  the next smallest $h$ value modulo the smallest, perform a GCD, and repeat  until we get an answer $\le k$, or use up all the $a_i$ base values.

\begin{example}  Let $k = 151121$, which is prime.  We have 
\begin{align*}
v( \mbox{ord}_k(2)) = 3  \implies 2^{37780} + 1 &\approx 8.189 \times 10^{11372},  \\
v( \mbox{ord}_k(3)) = 4 \implies 3^{75560} + 1 &\approx  1.914 \times 10^{36051},\\
v( \mbox{ord}_k(5)) = 0  \implies 5^{9445} - 1 &\approx  5.911 \times 10^{6601}.
\end{align*}
In this case, we would begin the computations with $5^{9445} - 1$ since it is the smallest. We then compute $\gcd(2^{37780}+1 \bmod (5^{9445}-1), 5^{9445}-1) = 151121$. This tells us that any $n$ of the form $151121\cdot p_t$ is not  a strong pseudoprime to the vector $\nu=(2,3,5)$.
\end{example}

These computations start off easy as $k$ is small and progressively get harder (in an essentially linear fashion) as $k$ grows.  As a practical matter, these computations only need to be done once.  As an example,  \cite{8bases} found 
\[ 84 \  98355 \ 74122 \ 37221 = 206135341 \cdot 412270681\]
to be a strong pseudoprime to the first eight prime bases.  When considering $p_1 = 206135341$, they sieved (which we explain in the next section) for $p_2 = 412270681$.  
It is reasonable to ask whether there is a larger prime that can be paired with $p_1$ that would form a strong pseudoprime to eight bases.  
In our computation we check by computing a GCD to answer ``no."

\begin{pseudocode}{Using GCDs to rule out $k$ values}
\INPUT Integers $t$ and $k$, where $t$ is the number of prime divisors  and $k$ is the product of $t-1$ primes with matching signature,  and the base vector $\nu$ containing the $m$ smallest primes.
\NUMLINE Let $b\in \nu$ give the smallest estimated value for $h(b,k)$  as defined above.
\NUMLINE If $b=a_1$, let $i=2$ else let $i=1$.
\NUMLINE Compute $h(b,k)$.
\NUMLINE Compute $x=h(a_i,k)\bmod h(b,k)$ using modular exponentiation.
\NUMLINE Set $x=\gcd(h(b,k),x)$;
\WHILE{ $x>k$ and $i<m$ }
\NUMLINE Set $i=i+1$; if $a_i=b$ set $i=i+1$ again.
\NUMLINE Compute $y=h(a_i,k)\bmod x$.
\NUMLINE Set $x=\gcd(x,y)$.
\ENDWHILE
\IF{$x < k$}
\NUMLINE We can rule out $k$.
\ELSE
\NUMLINE We factor $x$ and check each prime $p_t\mid x$ with $p_t>k$
  to see if $kp_t$ is a strong pseudoprime.
\ENDIF
\end{pseudocode}

Computing $h(b,k)$ is the bottleneck of this algorithm. When using $m=11$ or higher, we never found anything to factor  in the last step.

\subsection{Sieving\label{sec:sigsieve}}

As above, let $k = p_1 \cdots p_{t-1}$ and we try to construct $n = kp_t$ by finding $p_t$. 
In order to search up to a bound $B$, 
  we need to search primes $p_t$ in the interval $[ p_{t-1}, B/k]$.  
We combine two types of sieving to search this interval quickly.   
The first we call $\lambda$-sieving, and the second we call signature sieving.   
For each prime $p$ define 
\[ \lambda_p = \lcm \left\{ \mbox{ord}_p(a) : a \in \nu\right\}. \]
By the r-rank Artin Conjecture \cite{r_rank},   we expect with very high probability that $\lambda_p = p-1$.  Let $\lambda = \lcm \left\{ \lambda_{p_1}, \ldots, \lambda_{p_{t-1}} \right\}$. Since $a^{kp_t - 1} \equiv 1 \pmod{ n }$ for any $a \in \nu$ this means $kp_t - 1$ is a multiple of $\lambda$.  So, $p_t \equiv k^{-1} \pmod{ \lambda }$.  This tells us that $p_t$ lies in a specific residue class modulo $\lambda$,  and we call searching for $p_t$ using this value   \textit{$\lambda$-sieving}.  

We can further narrow down the residue classes that $p_t$ may lie in   with respect to primes in $\nu$.  The following two propositions appeared originally in \cite{jaeschke}.

\begin{Prop} \label{prop1}
For primes $p$ and $q$, if $v(p-1) = v(q-1)$ and $v( \mbox{ord}_p(a) ) = v( \mbox{ord}_q(a) ) $ then $\left( \frac{a}{p} \right) = \left( \frac{a}{q} \right)$.
\end{Prop}

\begin{Prop} \label{prop2}
For primes $p$ and $q$, if $v(p-1) < v(q-1)$ and $v( \mbox{ord}_p(a) ) = v( \mbox{ord}_q(a) ) $ then $\left( \frac{a}{q} \right) =1 $.
\end{Prop}

One may consider higher reciprocity laws,  and Jaeschke did so \cite{jaeschke}. We found quadratic reciprocity sufficient.  

The authors of \cite{8bases} consider searching for $p_t$ modulo $\lcm\{\lambda, 9240\}$.  They considered $30$ residue classes   modulo $9240 = 8 \cdot 3 \cdot 5 \cdot 7 \cdot 11$  that arose from the propositions above. With the inclusion of each additional prime from $\nu$,  we effectively rule out half of the cases to consider,  thereby doubling the speed of the sieving process.

Ideally, we would like to include every prime in $\nu$ for the best  performance, but with normal sieving, the sieve modulus becomes  quite large and requires we keep track of too many residue classes  to be practical. 
Instead, we adapted the space-saving wheel datastructure
  \cite{Sorenson10a, sorensonanother},  
  which had been used successfully to sieve for pseudosquares. 
Sieving for primes $p_t$ with specified quadratic character  modulo a list of small primes is the same algorithmic problem. 
This wheel sieve uses a data structure that takes space  proportional to $\sum_{a\in\nu} a$ instead of space   proportional to $\prod_{a\in\nu} a$ as used in traditional sieving. 
It does, however, produce candidate primes $p_t$ out of order.  
This is not an issue, so long as we make sure the sieve modulus  does not exceed $B$. 
In practice, we dynamically include primes from $\nu$ so that   $1000\cdot k \cdot \lcm\{ \lambda, 8, 3, 5, 7, 11, 13, 17, 19, 23, 29, 31\} < \psi_{13}$,  and we favor the inclusion of smaller primes.  

\begin{pseudocode}{Sieving}
\INPUT $k$ in factored form as $k=p_1p_2\cdots p_{t-1}$,  $\lambda_{p_i}$ for $i<t$,  search bound $B$, and base vector $\nu$ of length $m$.
\NUMLINE Compute $\lambda_k = \lcm\{ \lambda_{p_i} \}$.
\NUMLINE Set $w=1$. Compute wheel modulus $w$ as follows:
\FOR{ $i=2,\ldots,m$ }
\IF{ $k\lambda_ka_iw<B/1000$ }
  \IF{ $a_i$ does not divide $\lambda_k$ }
    \NUMLINE Set $w=w\cdot a_i$.
  \ENDIF
\ENDIF
\ENDFOR
\IF{ $\lambda_k$ divisible by 4 }
\NUMLINE Set $\lambda=\lambda_k$
\ELSE
\NUMLINE Set $\lambda=\lambda_k/2$ and
\NUMLINE Set $w=8w$
\ENDIF
\NUMLINE Let $\sigma$ be the signature of $p_1$.
\NUMLINE Build a wheel with modulus $w$  so that for each prime power $q\mid w$ we have  primes $p$ generated by the wheel with $(p/q)$ consistent with  $\sigma$ as in Propositions \ref{prop1} and \ref{prop2}.
\FOR{ each residue $r$ modulo $w$ generated by the wheel }
\NUMLINE Use the Chinese Remainder Theorem to compute  $x\bmod w\lambda$ such that $x\equiv r\bmod w$ and  $x\equiv k^{-1} \bmod \lambda$.
\NUMLINE Sieve for primes in the interval $[k+1,B/k]$  that are $\equiv x\bmod w\lambda$.
\NUMLINE For each such probable prime $p_t$ found,  perform a strong pseudoprime test on $kp_t$.
\ENDFOR
\end{pseudocode}

Note that if $k$ consists entirely of primes that are  $\equiv 3 \bmod 4$, then all primes $p_t$ found by the  sieve, when using all of $\nu$, will have the correct signature.

\begin{example}  Consider $p =3\cdot10^8 + 317 $.  Now, $\lambda  = p-1 = 2^2 \cdot 7 \cdot 11 \cdot 23 \cdot 42349$ (factorization is obtained via sieving).  We may use signature information for the primes $3$, $5$, $13$, $17$, $19$, $29$, $31$, and $37$.     However, since $B = 4\cdot 10^{24}$ we may only use the first 4 primes in the signature before the combined modulus is too large.   The signature information for 3 implies that $q \equiv 5, 7 \pmod{12}$ however, the case $q \equiv 7 \pmod{12}$ need not be considered since we know (because of $\lambda$) that $q \equiv 1 \pmod{4}$.  Given the signature information for the primes 5 and 13, we search in $\{2, 3 \pmod 5 \}$  and $\{2, 5, 6, 7, 8, 11 \pmod{ 13} \}$.    At this point we sieve for primes $q \equiv 1 \pmod{\lambda}$.  To this we add the following 12 new congruence conditions and sieve for each individual condition.  

\begin{enumerate}
\item Let $q \equiv 2 \pmod{ 5}$.
\begin{enumerate}
\item Let $q \equiv 2 \pmod{ 13}$.
\item Let $q \equiv 5 \pmod{ 13}$.
\item Let $q \equiv 6 \pmod{ 13}$.
\item Let $q \equiv 7 \pmod{ 13}$.
\item Let $q \equiv 8 \pmod{ 13}$.
\item Let $q \equiv 11 \pmod{ 13}$.
\end{enumerate}
\item Let $q \equiv 3 \pmod{ 5}$.
\begin{enumerate}
\item Let $q \equiv 2 \pmod{ 13}$.
\item Let $q \equiv 5 \pmod{ 13}$.
\item Let $q \equiv 6 \pmod{ 13}$.
\item Let $q \equiv 7 \pmod{ 13}$.
\item Let $q \equiv 8 \pmod{ 13}$.
\item Let $q \equiv 11 \pmod{ 13}$.
\end{enumerate}
\end{enumerate}

\end{example}

\subsection{Hash Table Datastructure\label{sec:datastructure}}

In order to compute GCDs or sieve as described above,  it is necessary to construct integers $k=p_1p_2\cdots p_{t-1}$  in an efficient way, where all the primes $p_i$ have matching  signatures. We do this by storing all small primes in a hash table datastructure  that supports the following operations:
  
\begin{itemize}
  \item Insert the prime $p$ with its signature $\sigma_p$.
    We assume any prime being inserted is larger than all primes
      already stored in the datastructure.    (That is, insertions are monotone increasing.)
    We also store $\lambda_p$ with $p$ and its signature for use later,
      thereby avoiding the need to factor $p-1$.
  \item Fetch a list of all primes from the datastructure,
     in the form of a sorted array $s[\,]$,
     whose signatures match $\sigma$.
    The fetch operation brings the $\lambda_p$ values along as well.
\end{itemize}

We then use this algorithm to generate all candidate $k$  values for a given $t$ and $\nu$:

\begin{pseudocode}{Generating $k$ Values}
\INPUT $t>2$, search bound $B$, cutoff $X<B$, base vector $\nu$ of length $m$.
\NUMLINE Let $T$ denote our hash table datastructure
\NUMLINE Let $a_{m+1}$ denote the smallest prime not in $\nu$.
\FOR{ each prime $p\le \sqrt{B/a_{m+1}^{t-2}}$ (with $p-1$ in factored form) }
\NUMLINE Compute $\lambda_p$ from the factorization of $p-1$;
\NUMLINE $s[\,] := T$.fetch($\sigma_p$);  List of primes with matching signature
\FOR{ $0\le i_1<\cdots<i_{t-2} \le s$.length }
\NUMLINE Form $k=s[i_1]\cdots s[i_{t-2}]\cdot p$;
\IF{ $k\le X$ }
\NUMLINE Use $k$ in a GCD computation to find matching $p_t$ values;
\ELSE
\NUMLINE Use $k$ for a sieving computation to find matching $p_t$ values;
\ENDIF
\ENDFOR
\IF{ $p \le ( B/a_{m+1}^{t-3} )^{1/3}$ }
\NUMLINE{ $T$.insert($p$,$\sigma_p$,$\lambda_p$); }
\ENDIF
\ENDFOR
\end{pseudocode}

Note that the inner for-loop above is coded using $t-2$  nested for-loops in practice. This is partly why we wrote separate programs for each value of $t$. To make this work as a single program that handles multiple $t$ values,  one could use recursion on $t$.

Also note that we split off the GCD and sieving computations into  separate programs, so that the inner for-loop was coded to make  sure we always had either $k\le X$ or $X<k<B/p$ as appropriate.

We implemented this datastructure as an array or table of  $2^m$ linked lists, and used a hash function on the prime signature  to compute which linked list to use. The hash function $h: \sigma \rightarrow 0..(2^m-1)$  computes its hash value from the signature as follows:

\begin{itemize}
  \item If $\sigma=(0,0,0,\ldots,0)$, the all-zero vector,    then $h(\sigma)=0$, and we are done.
  \item Otherwise, compute a modified copy of the signature, $\sigma^\prime$.    If $\sigma$ contains only 0s and 1s, then $\sigma^\prime=\sigma$.    If not, find the largest integer entry in $\sigma$,      and entries in $\sigma^\prime$ are 1 if the corresponding      entry in $\sigma$ matches the maximum, and 0 otherwise. 
   
    For example, if $\sigma=(3,4,0,4,2,1,2,4)$ (our example from      earlier), then $\sigma^\prime=(0,1,0,1,0,0,0,1)$.
  \item    $h(\sigma)$ is the value of $\sigma^\prime$ viewed as an integer in binary.
   
    So for our example, $h( (3,4,0,4,2,1,2,4) ) = 01010001_2 = 64+16+1=81$.
\end{itemize}

Computing a hash value takes $O(m)$ time.

Each linked list is maintained in sorted order;  since the insertions are monotone, we can simply append insertions  to the end of the list, and so insertion time is dominated by  the time to compute the hash value, $O(m)$ time. Note that the signature and $\lambda_p$ is stored with the prime  for use by the fetch operation.

The fetch operation computes the hash value, and then scans the linked  list, 
  extracting primes (in order) with matching signature. 
If the signature to match is binary (which occurs roughly half the time),  
  we expect about half of the primes in the list to match,  
  so constructing the resulting array $s[\,]$ takes time linear  
  in $m$ multiplied by the number of primes fetched. 
Note that $\lambda$ values are brought along as well. 
Amortized over the full computation, fetch will average time that  
  is linear in the size of the data returned. 
Also note that the average cost to compare signatures that don't match  is constant.

The total space used by the data structure is  an array of length $2^m$, plus one linked list node for each  prime $\le B^{1/3}$.  Each linked list node holds the prime, its $\lambda$ value,   and its signature, for a total of  $O( 2^m\log B + mB^{1/3})$ bits,  since each prime is $O(\log B)$ bits.

\subsection{Algorithm Outline}

Now that all the pieces are in place, we can  describe the algorithm as a whole.

Given an input bound $B$, we compute the GCD/sieve cutoff $X$  (this is discussed in \S\ref{analysis} below).

We then compute GCDs as discussed above using  candidate $k$ values $\le X$ for $t=2,3,\ldots$. For some value of $t$, we'll discover that there are no   candidate $k$ values with $t-1$ prime factors, at which point  we are done with GCD computations.

Next we sieve as discussed above, using  candidate $k$ values $>X$ for $t=2,3,\ldots$. Again, for some $t$ value, we'll discover there are no  candidate $k$ values with $t-1$ prime factors, at which point  we are done.

We wrote a total of eight programs;  GCD programs for $t=2,3,4$ and sieving programs  for $t=2,3,4,5,6$.   Also, we did not bother to use the wheel and signature sieving, relying only on $\lambda$ sieving, for $t>3$.

\section{Algorithm Analysis\label{analysis}}

\subsection{Model of Computation}

We assume a RAM model, with potentially infinite memory. Arithmetic operations on integers of $O(\log B)$ bits (single-precision)  and basic array indexing and control operations all are unit cost. We assume fast, FFT-based algorithms for multiplication and division  of large integers. For $n$-bit inputs, multiplication and division cost  $M(n)=O(n\log n \log\log n)$ operations \cite{SS71}. From this, an FFT-based GCD algorithm takes  $O(M(n)\log n)=O(n(\log n)^2\log\log n)$ operations;   see, for example, \cite{SZ2004}.

\subsection{Helpful Number Theory Results}

Let our modulus $q=\prod_{a_i\in \nu} a_i$ so that
   $\phi(q) = \prod_{a_i\in\nu} \phi(a_i)$;
   each base $a_i$ is an odd prime with $\phi(a_i)=a_1-1$,
   except for $a_1=8$ with $\phi(a_1)=4$. 
We use our two propositions from \S\ref{sec:sigsieve}, together with  quadratic reciprocity to
  obtain a list of $(a_i-1)/2$ residue  classes for each odd prime $a_i$
  and two residue classes modulo $8$.
The total number of residue classes is $2^{-m}\phi(q)$.
A prime we are searching for with a matching signature must lie in  one of these residue classes.

Let us clarify that when we state in sums and elsewhere 
  that a prime $p$ has signature \textit{equivalent} to $\sigma$, 
  we  mean that the quadratic character $(p/a_i)$ is consistent with $\sigma$ 
  as stated in Propositions \ref{prop1} and \ref{prop2}. 
We write $\sigma\equiv\sigma_p$ to denote this consistency under quadratic character, 
  which is weaker than $\sigma=\sigma_p$.

\begin{lemma} \label{dirichlet} 
  Given a signature $\sigma$ of length $m$,  
    and let $q=\prod_i a_i$ as above.
  Let $\epsilon>0$.
  If $q<x^{1-\epsilon}$, then
    number of primes $p\le x$ with $\sigma_p\equiv\sigma$  is at most
  \[  O\left( \frac{x}{2^m \log x} \right). \]
\end{lemma}
\begin{proof}
  This follows almost immediately from
   the Brun-Titchmarsh Theorem.
  See, for example, Iwaniec \cite{Iwaniec80}.
  We simply sum over the relevant residue classes mod $q$ based on quadratic character
    using Propositions \ref{prop1} and \ref{prop2}.
\end{proof}

We will choose $m$ proportional to $\log B/\log\log B$ so that
  $q=B^c$ for some constant $0<c<1$.

\begin{lemma} \label{lemma:sum1p}
We have 
\[ \sum_{p\le x, \sigma_p\equiv\sigma} \frac{1}{p}  = O( 2^{-m} \log\log x ). \]
\end{lemma}

This follows easily from Lemma \ref{dirichlet} by partial summation.
We could make this tighter by making use of
  Theorem 4 from \cite{LZ2007}, and observe that the constants $C(q,a)$ (in their notation)
  cancel with one another in a nice fashion when summing over
  different residue classes modulo $q$. See also section 6 of that paper.

Let $\pi_{t,\sigma}(x)$ denote the number of integers $\le x$  with exactly $t$ distinct prime divisors, each of which has  signature equivalent to $\sigma$ of length $m$.
\begin{lemma}\label{jslemma1}
\[ \pi_{t,\sigma}(x) \ll \frac{x (\log\log x)^{t-1} }{ 2^{tm} \log x} . \]
\end{lemma}

\begin{proof}
Proof is by induction on $t$.  $t=1$ follows from Lemma \ref{dirichlet}.

For the general case, we sum over the largest prime divisor $p$ of $n$,
\begin{eqnarray*}
\pi_{t,\sigma}(x)  
      &=& \sum_{ p\le x, \sigma_p=\sigma} \pi_{t-1,\sigma}(x/p) \\
      &\ll& \sum_{ p\le x, \sigma_p\equiv\sigma}
             \frac{x(\log\log (x/p))^{t-2}}{p 2^{(t-1)m}\log (x/p)} \\
      &\ll& \frac{x(\log\log x)^{t-1}}{2^{tm}\log x}.
\end{eqnarray*}
We used Lemma \ref{lemma:sum1p} for the last step. 
(See also the proof of Theorem 437 in \cite[\S22.18]{HW}.)
\end{proof}

Let $\lambda_{p,m}$ denote the value of $\lambda_p$  
  for the prime $p$  using a vector $\nu$ with the first $m$ prime bases. 
We need the following lemma from \cite[Cor. 2.4]{Pappalardi96}:

\begin{lemma} \label{lemma:pap}
There exists a constant $\tau>0$ such that
\[ \sum_{p\le x} \frac{1}{\lambda_{p,m}} \ll \frac{ x^{1/(m+1)}}{(\log x)^{1+\tau/(m+1)}}. \]
\end{lemma}

Using this, we obtain the following.

\begin{theorem} \label{thm:lambda}
Let $0<\theta<1$.  Then
\[ \sum_{x^\theta<p\le x} \frac{1}{p\lambda_{p,m}} \ll \frac{1}{ x^{\theta - 1/(m+1)} \log x }. \]
\end{theorem}

\begin{proof}
  We have
\begin{eqnarray*}
  \sum_{x^\theta<p\le x} \frac{1}{p\lambda_{p,m}}
    &\le& \frac{1}{x^{\theta}} \sum_{x^\theta<p\le x} \frac{1}{\lambda_{p,m}}  \\
    &\ll&   \frac{1}{x^{\theta}} \frac{ x^{1/(m+1)}}{(\log x)^{1+\tau/(m+1)}} .
\end{eqnarray*}
Note that $\tau>0$.
\end{proof}

\subsection{A Conjecture}

In order to prove our $B^{2/3+o(1)}$ running time bound for $t>2$,
  we make use of the conjecture below.
Let ${\mathcal{K}}={\mathcal{K}}(t,\sigma)$ be the set of integers
  with exactly $t$ distinct prime divisors each of which has signature
  matching $\sigma$.
In other words, $\pi_{t,\sigma}(x)$ counts integers in $\mathcal{K}$
  bounded by $x$.

\begin{conjecture} \label{conjecture}
  Let $B,\sigma,t,m,\nu$ be as above.
  Then
  $$
    \sum_{k\le x, k\in{\mathcal{K}}(t,\sigma)}
       \frac{1}{\lambda_{k,m}} \ll x^{o(1)}.
  $$
\end{conjecture}

\begin{theorem} \label{thm:lambda2}
  Assuming Conjecture \ref{conjecture} is true, we have
  $$
    \sum_{x^\theta \le k\le x, k\in\mathcal{K}(t,\sigma)}
       \frac{1}{k\lambda_{k,m}} \ll x^{-\theta+o(1)}.
  $$
\end{theorem} 
\begin{proof}
  The proof is essentially the same as that of Theorem \ref{thm:lambda}.
  Simply substitute Conjecture \ref{conjecture} for Lemma \ref{lemma:pap}.
\end{proof}

\subsection{Proof of Theorem \ref{thm:runtime}}

We now have the pieces we need to prove our running time bound.

As above, each pseudoprime candidate we construct will have the form  $n=kp_t$, where $k$ is the product of $t-1$ distinct primes  with matching signature. Again, $\nu=(2,3,5,\ldots)$ is our base vector of small primes  of length $m$.

\subsubsection{$t=2$}  In this case $k$ is prime.

\textbf{GCD Computation.} For each prime $k=p\le X$ we perform a GCD computation as  described in \S\ref{sec:gcd}.  The bottleneck of this computation is computing GCDs of  $O(p)$-bit integers. This gives a running time proportional to
\begin{eqnarray*}
  \sum_{p\le X} M(p)\log p &\ll& \pi(X) X(\log X)^2 \log\log X \\
            &\ll& X^2 \log X \log\log X.
\end{eqnarray*}

\textbf{Sieving Computation.} For each prime $k=p$ with $X<p<B/p$, we sieve the interval $[p+1,B/p]$  for primes that are $\equiv 1 \bmod \lambda_p$. We also employ signature sieving, but for the simplicity of analysis,  we omit that for now. Using the methods from \cite{sorensonanother}, we can sieve an  arithmetic progression of length $\ell$ in $O(\ell\log B)$  operations. We do not need proof of primality here, so a fast  probabilistic test works fine. This gives a running time proportional to
\[  \sum_{X< p\le B/p} \frac{ B\log B }{p\lambda_p} \ll B\log B \sum_{X< p\le B/X} \frac{ 1 }{p\lambda_p} .\]
At this point we know $X^2\le B$ and, to keep the GCD and sieving  computations balanced, $X\ge B^{1/4}$, say. This means Theorem \ref{thm:lambda} applies; we set  $x^\theta =X$ and $x=B/X$ to obtain
\begin{eqnarray*}
  B\log B \sum_{X< p\le B/X} \frac{ 1 }{p\lambda_p}  &\ll&    B\log B \frac{ (B/X)^{1/(m+1)} }{ X \log (B/X) } \\
    &\ll&     \frac{ B }{X} (B/X)^{1/(m+1)}   \\
    &=& (B/X) ^ {1+o(1)}
\end{eqnarray*}
assuming that $m\rightarrow \infty$ with $B$.

Minimizing $ X^{2+o(1)} + (B/X)^{1+o(1)}$ implies $X=B^{1/3}$  and gives the desired running time of  $B^{2/3+o(1)}$.  This completes our proof.
\hfill $\qed$.

\subsubsection{$t>2$}  In this case $k$ is composite.

\textbf{GCD Computation.}
  For $t>2$ we construct integers $k=rp$ for computing GCDs,
  with $r$ consisting of exactly $t-2$ prime divisors less than $p$,
  with signatures matching $\sigma_p$.
For each prime $p$, we perform a hash table lookup and fetch the list of such primes;
  this is Step 5 of Algorithm 3.
This overhead cost is $O(m)+O(\pi_{\sigma_p}(p))=O(m+ 2^{-m}p/\log p)$. 
Summing this over all primes $p\le X$ gives  $O( 2^{-m} X^2/\log X)$. 

Next, for each prime $p\le X$ we construct at most $\pi_{t-2,\sigma_p}(X/p)$ values for $r$
  (Step 7 of Algorithm 3),
  and using Lemma \ref{jslemma1} and multiplying by the cost of computing the GCD, 
  this gives us
\begin{equation}
  \sum_{p\le X} \pi_{t-2,\sigma_p}(X/p) 
     \cdot M(X)\log X
   \ll  X^2  \left(\frac{\log\log X }{2^{m}}\right)^{t-2} \log X  \label{t2gcd}
\end{equation}
for the total running time.

\textbf{Sieving Computation.}
Again, the main loop enumerates all choices for the second-largest prime $p=p_{t-1}$. 
First, we construct all possible $k$-values with $k>X$, $k<B/p$, and $p\mid k$,  
  so $k/p=r=p_1\cdots p_{t-2}$ with all the $p_i<p_{i+1}$ and $p_{t-2}<p$. 
We also have $p_{t-2}<B^{1/3}$. 
This implies $p>X^{1/(t-1)}$.

For a given $p$, fetching the list of primes below $u:=\min\{ p,B/p^2\}$
  with matching signatures takes $O(m+u/(2^m\log u))$ time
  (Algorithm 3 Step 5).
Summing over all $p\le \sqrt{B}$, splitting the sum at $p=B^{1/3}$,
  this takes $2^{-m} B^{2/3+o(1)}$ time.

As above, we write $r=k/p$.
We claim that the total number of $k$ values is $2^{-m(t-2)} B^{2/3+o(1)}$.
Let $u$ be as above.
There are at most $\pi_{t-2,\sigma_p}(u)$ values of $r$ for each $p$.
Again, splitting this at $B^{1/3}$, we have
\begin{eqnarray*}
  \sum_{X^{1/(t-1)}<p\le B^{1/3}} \pi_{t-2,\sigma_p}( p ) 
  &\ll&  \sum_{X^{1/(t-1)}<p\le B^{1/3}}
      \frac{ p (\log\log p)^{t-3} }{ 2^{(t-2)m}\log p} \\
  &\ll& 2^{-m(t-2)} B^{2/3+o(1)}.
\end{eqnarray*}
and
\begin{eqnarray*}
  \sum_{B^{1/3}<p\le \sqrt{B}} \pi_{t-2,\sigma_p}( B/p^2 ) 
  &\ll&  \sum_{B^{1/3}<p\le \sqrt{B}}
     \frac{ B (\log\log B)^{t-3} }{p^2 2^{(t-2)m}\log B} \\
  &\ll&    \frac{1}{ B^{1/3} \log B}
     \frac{ B (\log\log B)^{t-3} }{2^{(t-2)m}\log B} \\
  &\ll& 2^{-m(t-2)} B^{2/3+o(1)}.
\end{eqnarray*}
This covers work done at Step 7 of Algorithm 3.

Finally, the cost of sieving is at most
\begin{eqnarray*}
  \sum_{X^{1/(t-1)}<p\le \sqrt{B}} \sum_{r} \frac{ B\log B }{ rp\lambda_{rp} } .
\end{eqnarray*}
If we use Theorem \ref{thm:lambda2}, based on our conjecture,
  and using the lower bound $X\le k=rp$, this leads to the bound
\begin{equation}
  \left( \frac{ \log\log B }{2^m} \right)^{t-2}  \frac{B^{1+1/(m+1)+o(1)}}{X}. 
  \label{t2sieveA}
\end{equation}
Without the conjecture,
  we use Lemma \ref{lemma:sum1p} together with the argument outlined
  in Hardy and Wright \cite{HW} in deriving (22.18.8),  we obtain that
\[  \sum_{r} \frac{1}{r} \sim \left( \frac{ \log\log B }{2^m} \right)^{t-2}. \]
We then use
$$
  \frac{1}{\lambda_{rp}} \le \frac{1}{\lambda_p}.
$$
As above, this leads to the bound
\begin{equation}
  \left( \frac{ \log\log B }{2^m} \right)^{t-2}  \frac{B^{1+1/(m+1)}}{X^{1/(t-1)}}. 
  \label{t2sieveB}
\end{equation}

To balance the cost of GCDs with sieving, we want to balance
  (\ref{t2gcd}) with (\ref{t2sieveA}) or (\ref{t2sieveB}),
  depending on whether we wish to assume our conjecture or not.
Simplifying a bit, this means balancing
  $X^{2+o(1)}$ with either $(B/X)^{1+o(1)}$ or $B^{1+o(1)}/X^{1/(t-1)}$.
The optimal cutoff point is then
  either $X=B^{1/3}$ under the assumption of Conjecture \ref{conjecture},
  or $X=B^{ ({t-1})/({2t-1}) }$ unconditionally,
  for a total time of
\[ \left( \frac{\log\log B }{2^{m}} \right)^{t-2} 
      B^{2/3+o(1)} \]
with our conjecture, or
\[ \left( \frac{\log\log B }{2^{m}} \right)^{t-2}
     B^{1- \frac{1}{2t-1} +o(1)} \]
without.
We have proven the following.

\begin{theorem}
  Assuming Conjecture \ref{conjecture},
  our algorithm takes $B^{2/3+o(1)}$ time to find all integers $n\le B$
  that are strong pseudoprimes to the first $m$ prime bases,
  if $m\rightarrow\infty$ with $B$.
\end{theorem}

If we choose $m$ so that $q$ is a fractional power of $B$,
  then Lemma \ref{jslemma1} implies that there is a constant $c>0$
  such that if $t>c\log\log B$, then there is no work to do.
This explains why our computation did not need to go past $t=6$.
It also explains why, in practice, the $t=2$ case is the bottleneck
  of the computation,
  and is consistent with the implications of our conjecture
  that the work for each value of $t$ decreases as $t$ increases.

\section{Implementation Details}

In this final section, we conclude with some details from our algorithm implementation.

\subsection{Strong Pseudoprimes Found}

In addition to our main results for $\psi_{12}$ and $\psi_{13}$,
  we discovered the following strong pseudoprimes.
This first table contains pseudoprimes found while
  searching for $\psi_{12}$
\begin{center}
\begin{tabular}{| r | l |}
\hline
$n$  & Factorization \\ \hline
3825123056546413051& 747451$\cdot$ 149491 $\cdot$34233211 \\ \hline
230245660726188031& 787711$\cdot$ 214831$\cdot$ 1360591 \\ \hline
360681321802296925566181& 424665351661$\cdot$849330703321 \\ \hline
164280218643672633986221 & 286600958341$\cdot$573201916681 \\ \hline
 318665857834031151167461& 399165290221$\cdot$798330580441 \\ \hline
7395010240794120709381& 60807114061$\cdot$121614228121 \\ \hline
164280218643672633986221 & 286600958341$\cdot$573201916681 \\ \hline
\end{tabular}
\end{center}
This second table contains pseudoprimes found while verifying $\psi_{13}$.  
\begin{center}
\begin{tabular}{| r | l |}
\hline
$n$  & Factorization \\ \hline
318665857834031151167461 & 399165290221$\cdot$798330580441 \\ \hline
2995741773170734841812261& 1223875355821$\cdot$2447750711641 \\ \hline
667636712015520329618581& 577770158461$\cdot$1155540316921 \\ \hline
3317044064679887385961981& 1287836182261$\cdot$2575672364521 \\ \hline
3317044064679887385961981&1247050339261$\cdot$2494100678521 \\ \hline
552727880697763694556181& 525703281661$\cdot$1051406563321 \\ \hline
360681321802296925566181& 424665351661$\cdot$849330703321 \\ \hline
7395010240794120709381& 60807114061$\cdot$121614228121 \\ \hline
3404730287403079539471001& 1304747157001$\cdot$2609494314001 \\ \hline
164280218643672633986221& 286600958341$\cdot$573201916681 \\ \hline
\end{tabular}
\end{center}

\subsection{Hash Table Datastructure}

It is natural to wonder how much time is needed to manage the hash table datastructure
  from \S\ref{sec:datastructure}.
Specifically, we measured the time to create the table when inserting all primes up to
  a set bound.
For this measurement, we used a vector $\nu$ with the first 8 prime bases.
We also measured how long it took, for each prime up to the bound listed, to also perform
  the fetch operation, which returns a list of primes with matching signature with $\lambda$
  values.
\begin{center}
\begin{tabular}{| c | c | c | c |}
\hline
Largest Prime  & Table Entries & Creation Time & Fetching Time \\ \hline
$10^6$ & 78490 & 1.13  & 2.58  \\ \hline
$10^7$ & 664571 &11.49 & 2:44.49 \\ \hline
$10^8$ & 5761447 &   1:56.22  & 3:30:50.2 \\ \hline 
\end{tabular}
\end{center}
The times are given in seconds, or minutes:seconds, or hours:minutes:seconds as appropriate.

We felt the data structure was a success, in that it was fast and small enough that
  it was not noticed in the larger computation.

\subsection{Almost-Linear GCDs}

\begin{figure}
\centering
\includegraphics{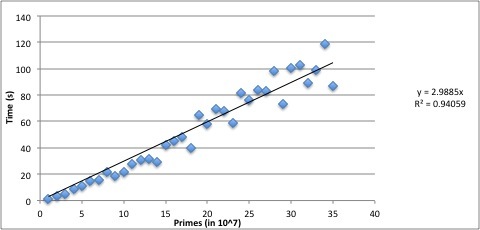}
\caption{GCD timings}
\label{gcd-timing}
\end{figure}

In Figure \ref{gcd-timing},
  we present the times to perform Algorithm 1 on selected prime values for $k$ 
  (so this is for the $t=2$ case).
The data points were obtained by averaging the time
  needed to carry out Algorithm 1 for the first ten primes 
  exceeding $n\cdot10^7$ for $1 \leq n \leq 35$.   
Averaging was required because the size of the $h$ values is not uniformly increasing in $k$.  
This explains the variance in the data; for example, the first ten primes after $35\cdot 10^7$ 
  happened to have smaller $h$ values on average than the first ten primes after $34\cdot10^7$.  

It should be clear from the data that our GMP algorithm for computing GCDs
  was using an essentially linear-time algorithm.  
Note that if we chose to extend the computation to large enough $k$ values, memory would become a
  concern and we would expect to see paging/thrashing degrade performance.

\subsection{GCD/Sieving Crossover}
\begin{figure}
\centering
\includegraphics[width=4in]{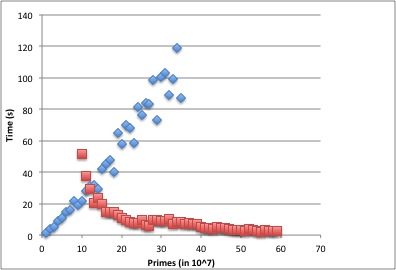}
\caption{Signature Sieving and GCD timing}
\label{crossover}
\end{figure}

In Figure \ref{crossover}, we present a comparison of the timings for computing GCDs
  (Algorithm 1, diamonds on the graph) with signature sieving (Algorithm 2, squares on the graph).
For this graph, we are looking at choosing the crossover point $X$,
  and again, we are focusing on the $t=2$ case.
\begin{itemize}
\item
  One would expect that the signature sieving curve should behave as an inverse quadratic;
    the time to sieve for one prime $k=p$ is proportional to
    $$ \frac{B\log B}{p\lambda_p}$$
    and we expect $\lambda_p\approx p-1$ most of the time.
  However, utilizing signatures obscures this to some degree, since the algorithm cannot
    make use of all prime bases if they divide $\lambda$, hence the minor variations in the curve.
  Let us elaborate this here.

The data points for signature sieving in Figure \ref{crossover} represent the average sieve time for the first 50 primes past $n\cdot10^7$ for $10 \leq n \leq 60$.  There is a lot of variance in these timing because of the inclusion and exclusion of primes from the signature depending if they are relatively prime to $\lambda$.   For example, $p = 174594421$ has $\lambda = 2^2\cdot3^3\cdot5\cdot11\cdot13\cdot17\cdot19$ and therefore has few primes from the base set to use in signature sieving.  The inability to add extra primes shows up in the timing data;  a careful inspection of the data shows a strange jump when primes transition from $28\cdot10^7$ to $29\cdot10^7$.  The data points to a steady decrease, then an almost two fold increase in time followed by a steady decrease again.  We believe this is because, on average, one less base prime may be used in signature sieving.  Using one less base prime results in approximately twice as much work.  

\item
  In the computation for $\psi_{12}$, our actual crossover point $X$ was $30\cdot10^7$,
  even though the timing data would suggest the optimal crossover point is around $12\cdot10^7$.
  From a cost-efficiency point of view, we chose poorly.  
  However, the choice was made with two considerations.  
  One is that the program to compute GCDs (Algorithm 1) is simple to write,
    so that program was up and running quickly, and we let it run
    while we wrote the sieving code (Algorithm 2).
  Second is that, initially, we did not know which $\psi_m$ value we would ultimately be able 
    to compute.  
  Since the results from  GCD computations apply to all larger values of $m$,
    we opted to err in favor of computing more GCDs.  
\end{itemize}

\subsection{Signature Sieving}

\begin{figure}
\centering
\includegraphics[width=4in]{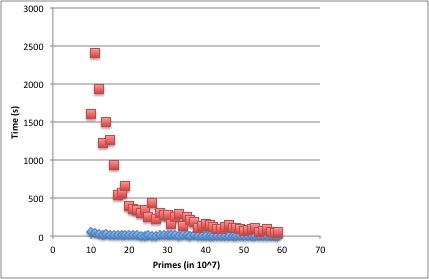}
\caption{Signature Sieving and $\lambda$-Sieving}
\label{sievecomp}
\end{figure}

Figure \ref{sievecomp} shows the impact that signature sieving makes.  
Here the squares in the graph give $\lambda$-sieving times with no signature information used,
  and the diamonds show the same $\lambda$-sieving work done while taking advantage of
  signature information using the space-saving wheel.
Since there is relatively little variance involved in $\lambda$-sieving, each point represents the first prime after $n\cdot10^7$ for $10 \leq n \leq 60$.   On the same interval signature sieving practically looks like it takes constant time compared to  $\lambda$-sieving.  If we had opted to not incorporate signature sieving in, then the expected crossover point given these timings would occur when the primes are around $42\cdot10^7$.  


\bibliographystyle{amsplain}
\providecommand{\bysame}{\leavevmode\hbox to3em{\hrulefill}\thinspace}
\providecommand{\MR}{\relax\ifhmode\unskip\space\fi MR }
\providecommand{\MRhref}[2]{%
  \href{http://www.ams.org/mathscinet-getitem?mr=#1}{#2}
}
\providecommand{\href}[2]{#2}

\end{document}